\documentclass[reqno]{amsart}
\usepackage{amsmath,cite,amsfonts,amssymb,amsthm,amscd,latexsym,cite}
\newcommand{\conv}{\mathrm{conv}}
\usepackage{array}
\usepackage{mathrsfs}

\usepackage{comment}
 \usepackage{lmodern}
\usepackage{a4wide}
\newtheorem{thm}{Theorem}[section]
\newtheorem{theorem}[thm]{Theorem}

\newtheorem{lemma}[thm]{Lemma}
\newtheorem{cor}[thm]{Corollary}
\newtheorem{corollary}[thm]{Corollary}
\newtheorem{proposition}[thm]{Proposition}

\newtheorem{definition}[thm]{Definition}
\newtheorem{remark}[thm]{Remark}

\newtheorem{question}{Question}
\newtheorem{example}[thm]{Example}
\newcommand{\cA}{{\mathcal A}}
\newcommand{\cB}{{\mathcal B}}

\newcommand{\cF}{{\mathcal F}}

\newcommand{\cH}{{\mathcal H}}

\newcommand{\cM}{{\mathcal M}}

\newcommand{\cP}{{\mathcal P}}

\newcommand{\cU}{{\mathcal U}}

\newcommand{\cZ}{{\mathcal Z}}

\usepackage{graphicx}

 \usepackage{color}
 \newcommand{\norm}[1]{\left\lVert#1\right\rVert}
\usepackage{mathtools}
\usepackage{hyperref}                   
\hypersetup{colorlinks,
    linkcolor=blue,%
    citecolor=blue}
\usepackage{etoolbox}

 \usepackage{float}

\begin{document}

\title[Derivations  into noncommutative symmetric spaces]{Innerness of derivations into noncommutative symmetric spaces is determined commutatively}

\author[J. Huang]{Jinghao Huang}
\address{Institute for  Advanced Study in  Mathematics of HIT, Harbin Institute of Technology, Harbin, 150001, China}
\email{{\color{blue}jinghao.huang@hit.edu.cn}}

\author[F. Sukochev]{Fedor Sukochev}
\address{School of Mathematics and Statistics, University of New South Wales, Kensington, 2052, Australiaa}
\email{{\color{blue}f.sukochev@unsw.edu.au}}

\begin{abstract}
Let $E=E(0,\infty)$ be a symmetric function space and $E(\cM,\tau)$ be a symmetric operator space associated with a semifinite von Neumann algebra with a faithful normal semifinite trace.  Our main result identifies the class of spaces $E$ for which every
derivation $\delta:\cA\to E(\cM,\tau)$ is necessarily inner for each $C^*$-subalgebra $\cA$ in the class of all  semifinite von Neumann algebras $\cM$ as those with the Levi property.
%
\end{abstract}


\thanks{J. Huang was supported by the NNSF of China (No.12031004).
F. Sukochev's research was supported by the  Australian Research Council  (FL170100052).}
 \maketitle

\section{Introduction}

Let $\cA $ be a $C^*$-algebra and let $J$ be an $\cA$-bimodule\cite{Sinclair_S}.
A derivation $\delta :\cA \rightarrow J$ is a linear mapping satisfying $\delta (xy) = \delta(x)y + x\delta(y)$, $x,y \in \cA $.
In particular, if $a\in J$, then $\delta_a(x): = xa-ax $ is a derivation.
Such derivations implemented by elements in $J$ are said to be \emph{inner}.
One of the classical problems in operator algebra theory is the question whether every derivation from $\cA$ into $J$ is automatically inner.

The celebrated Kadison--Sakai theorem \cite{Kadison,Sakai66}
states that derivations are always inner  when $\cA$ is a von Neumann algebra and the $\cA$-bimodule $J$ coincides with the algebra $\cA$ itself.
Further, it was proved that every derivation from a von Neumann algebra into any of its ideal s is automatically inner \cite{Ber_S_1,Ber_S_2}.
However, when one considers more general $C^*$-algebras $\cA$ and $\cA$-bimodules $J$, there are  examples of non-inner derivations for some specific $\cA$ and $J$
(see e.g. \cite{Popa_R,BCS2006,Sakai98,BHLS,Elliott2,Elliott1,Elliott3}).
We  quote the following  from the memoir by Johnson \cite[Section 10.11]{Johnson}): ``\emph{It would be desirable to identify those spaces $\mathfrak{X}$ with $H^1(G, \mathfrak{X}^*)= 0$
for all $G$ }''.
A similar statement appeared in \cite{Hoover} (see also \cite[p.60]{Sinclair_S}): ``\emph{Here again it can be asked if such derivations
are inner; that is, are they induced by an element of $J$ as above? In fancier
language, the question asks if the cohomology group $H^1(\mathfrak{A},J)$ is trivial.}''
   During the past decades,   a number of important special cases have been studied (see e.g. \cite{CPSS,Pisier,Christensen_E_S,Johnson, Johnson_P,Popa,BGM,Losert,Runde,BKS22,KL1,KL2,Davidson}).

Due to the rapid development of noncommutative analysis and motivated by questions due to  Johnson et al.,
there are a number of papers concerning various versions of  the following question\cite{BP,BCS_2014,BGM,Weigt,BCLSZ}: 
 \begin{question}\label{que:1}
Assume that $\cM$ is a  von Neumann algebra  equipped with a faithful normal semifinite trace $\tau$.
Let  $E(\cM,\tau)$ be a symmetric space of $\tau$-compact operators  affiliated with  $\cM$.
How can one identity those  $E(\cM,\tau)$ such that derivations from
 an arbitrary  $C^*$-subalgebra $\cA$ of $\cM$  into $E(\cM,\tau)$ are necessarily inner?
\end{question}

Experts in the operator theory are probably     more familiar with symmetrically normed ideals in $B(\cH)$, which are a special case of noncommutative symmetric spaces.
Various versions of  Question \ref{que:1}
  for derivations with values in ideals of a von Neumann algebra were asked and discussed  in \cite{Kaftal_W,BHLS,BHLS2,Johnson_P,Popa, Popa_R,Huang, Hoover}.

It was a long-standing open question whether every derivation a $C^*$-subalgebra of a semifinite von Neumann algebra $\cM$ into $\cM_*$ must be  inner (see e.g. \cite[p.247]{BP}),
 which in our setting is equivalent to the special version of Question \ref{que:1} with $E=L_1$.
This question  was resolved completely by Bader, Gelander and Monod in 2012 \cite{BGM} (see also \cite{Pfitzner} for a slightly different proof due to Pfitzner).
The method used in \cite{BGM} (or \cite{Pfitzner}) does not have any chance to deliver the full answer on Question \ref{que:1}.
This fact was emphaiszed in  \cite[Section 3]{BGM}, where the following points were raised:
      \begin{quote}
    \begin{enumerate}
      \item[a.] In marked contrast to the classical fixed point theorems, there is no hope to find a fixed point inside a general bounded closed convex subset of $L^1$ $\cdots$ the weak compactness $\cdots$ seems almost unavoidable $\cdots$
      \item[b.]
$\cdots$ a canonical norm one projection $V^{**}\to V$ is not enough.
\item[c.] It would be interesting to find a purely geometric version of the proposition $\cdots$
\end{enumerate}
\end{quote}
The fact that the ``fixed point'' obtained in \cite{BGM} is not inside a general bounded closed convex subset of $L^1$ leads to extra difficulties in the general case.

In this paper,  we completely resolve  Question \ref{que:1} above, see Theorem \ref{main} and Corollary \ref{cormain} below.
We show that the Levi property\footnote{A symmetric function space $E(0,\infty)$ having the Levi property has an equivalent symmetric norm such that $E(0,\infty)$ has the Fatou property, see Remark \ref{Levi}.
The Soviet school on
Banach lattices used the term monotone complete norm or property (B)\cite[Chapter X.4]{KA}, see also \cite[p.89]{AA}. In the theory of operator algebras,  a similar property is called `\emph{monotone closed}'\cite[Chapter III, Definition 3.13]{Tak}.}
of a symmetric space is a sufficient and necessary condition for the Question \ref{que:1} above having an affirmative answer for every  semifinite von Neumann algebra $\cM$.
   Thus, the derivation theorem for preduals of von Neumann algebras in \cite{BGM} (in the semifinite setting) becomes a trivial corollary of our Theorem \ref{main}.

The new approach devised in this paper answers to
  points (a), (b) and (c) above raised in \cite{BGM}, which provides  an alternative proof for the resolution of the question raised by Bunce and Paschke~\cite{BP}, without  involving weak compactness of a subset in a $L$-embedded space  as \cite{BGM}   and  \cite{Pfitzner} did.
  This
enables us to find
a  ``fixed point'' (implementing the derivation)  from  a  not necessarily weakly compact  closed convex subset of  a noncommutative symmetric space  which is  $1$-complemented subspace of its bidual but  not necessarily an $L$-embedded Banach space (i.e., a symmetric space having the Levi property\cite{DP2012,DDST}\footnote{Indeed,
Theorem \ref{main} holds for the case when
the projection constant   is not necessarily $1$.
}),
  see  Theorem~\ref{main}.
On the other hand, the Levi property
 of the space $E(\cM,\tau)$ means that $E(\cM,\tau)$ coincides with its second K\"{o}the dual and this geometrical condition is the only one required in Theorem \ref{main} thus delivering
   (at least spiritually)   an answer to the question suggested  in  \cite[Comment c]{BGM} above.
 We  believe that the method   developed in this
work is of interest in its own right.

It is important to emphasize that the Fatou/Levi property was hatched in the theory of Banach lattices\cite{BVG,AA}, and was even included into the original
  definition of  Banach function spaces  over $\sigma$-finite measure spaces (see \cite{Bennett_S,Luxemburg}). The property
 is somewhat analogous to the   so-called ``dual normal'' property\footnote{Let $\cM$ be a von Neumann algebra. An $\cM$-bimodule  $X$ is said to be a dual normal $X$-bimodule if $X$ is a dual space and the maps $m\mapsto mx$ and $m\mapsto xm$ are both ultraweak-weak$^*$ continuous from $\cM$ into $X$ for each fixed element $x\in X$\cite[p.6]{Sinclair_S}. }.
The importance of the Fatou/Levi property in the theory of Banach function spaces  and symmetric operator spaces is hard to overestimate  \cite{DDP93,DP2012,DDST}.
It seems appropriate to recall here that   every derivation from a hyperfinite  von Neumann algebra $\cA$ into a dual normal $\cA$-bimodule is inner
(see e.g. \cite[Theorem 2.4.3]{Sinclair_S}, \cite{Connes} and \cite{JKR}).
Recall also, that derivations from a nuclear $C^*$-algebra $\cA$ into a dual Banach $\cA$-module are inner\cite{Haagerup,Connes78}.
  However,   Theorem~\ref{main} below
  holds for arbitrary  $C^*$-subalgebras $\cA$ of $\cM$
     and for symmetric spaces which may  not  have a predual space. 


The  reflexive gate type result (see e.g. \cite[Corollary 3.2.3]{Novikov})
is relatively unknown but plays a significant role in our approach. 
      In Section \ref{RGT}, we establish
       a noncommutative version  of
       this result and lay the groundwork for its usage in the derivation problem.
 In Section \ref{s:M}, using the weak compactness criteria for noncommutative symmetric spaces obtained in \cite{DDP,DSS}, we show that the Ryll-Nardzewski fixed point theorem is applicable to any noncommutative strongly symmetric $KB$-space (or a Kantorovich--Banach space, see Section \ref{sub:sys}) whose bounded part does not coincides with $C_1(\cM,\tau)=L_1(\cM,\tau)\cap \cM$.
This is the key  technical step in the proof of  the main result of the present paper, Theorem \ref{main}.
In Section \ref{iffsection}, we demonstrate why the Fatou/Levi property of the space $E(\cM,\tau)$  is a   necessary condition for an affirmative answer to Question \ref{que:1} in the class of all semifinite von Neumann algebras.


%
%
 %

\section{Preliminaries}\label{prel}
In this section,
we recall some notions of the theory of noncommutative integration.

In what follows,  $\cH$ is a  Hilbert space and $B(\cH)$ is the
$*$-algebra of all bounded linear operators on $\cH$ equipped with the uniform norm $\left\|\cdot\right\|_\infty$, and
$\mathbf{1}$ is the identity operator on $\cH$.
Let $\mathcal{M}$ be
a von Neumann algebra on $\cH$.
We denote by $\cP(\cM)$ the collection  of all projections in $\cM$, by $\cM'$ the commutant of $\cM$ and by $\cZ(\cM)$ the center of $\cM$.
For details on von Neumann algebra
theory, the reader is referred to e.g.   \cite{KR1, KR2}
or \cite{Tak}. General facts concerning measurable operators may
be found in \cite{Nelson}, \cite{Se} (see also
\cite[Chapter
IX]{Ta2} and the forthcoming book \cite{DPS}).
For convenience of the reader, some of the basic
definitions are recalled.

\subsection{$\tau$-measurable operators and generalized singular value functions}

A closed, densely defined operator $x:\mathfrak{D}\left( x\right) \rightarrow \cH $ with the domain $\mathfrak{D}\left( x\right) $ is said to be {\it affiliated} with $\mathcal{M}$
if $yx\subseteq xy$ for all $y\in \mathcal{M}^{\prime }$, where $\mathcal{M}^{\prime }$ is the commutant of $\mathcal{M}$.
A  closed,
densely defined
operator $x:\mathfrak{D}\left( x\right) \rightarrow \cH $ affiliated with $\cM $ is said to be
{\it measurable}  if  there exists a
sequence $\left\{ p_n\right\}_{n=1}^{\infty}\subset \cP\left(\mathcal{M}\right)$, such
that $p_n\uparrow \mathbf{1}$, $p_n(\cH)\subseteq\mathfrak{D}\left(x\right) $
and $\mathbf{1}-p_n$ is a finite projection (with respect to $\mathcal{M}$)
for all $n$.
 The collection of all measurable
operators with respect to $\mathcal{M}$ is denoted by $S\left(
\mathcal{M} \right) $, which is a unital $\ast $-algebra
with respect to strong sums and products (denoted simply by $x+y$ and $xy$ for all $x,y\in S\left( \mathcal{M%
}\right) $).

Let $x$ be a self-adjoint operator affiliated with $\mathcal{M}$.
We denote its spectral measure by $\{e^x\}$.
 It is well known that if
$x$ is an operator affiliated with $\mathcal{M}$ with the
polar decomposition $x = u|x|$, then $u\in\mathcal{M}$ and $e\in
\mathcal{M}$ for all projections $e\in \{e^{|x|}\}$. Moreover,
$x\in S(\mathcal{M})$ if and only if  $e^{|x|}(\lambda,
\infty)$ is a finite projection for some $\lambda> 0$. It follows
immediately that in the case when $\mathcal{M}$ is a von Neumann
algebra of type $III$ or a type $I$ factor, we have
$S(\mathcal{M})= \mathcal{M}$. For type $II$ von Neumann algebras,
this is no longer true. From now on, let $\mathcal{M}$ be a
semifinite von Neumann algebra equipped with a faithful normal
semifinite trace $\tau$.

An operator $x\in S\left( \mathcal{M}\right) $ is called $\tau$-measurable if there exists a sequence
$\left\{p_n\right\}_{n=1}^{\infty}$ in $P\left(\mathcal{M}\right)$ such that
$p_n\uparrow \mathbf{1}$, $p_n(\cH)\subseteq \mathfrak{D}\left(x\right)$ and
$\tau(\mathbf{1}-p_n)<\infty $ for all $n$.
The collection $S\left( \mathcal{M}, \tau\right)
$ of all $\tau $-measurable
operators is a unital $\ast $-subalgebra of $S\left(
\mathcal{M}\right) $.
It is well known that a linear operator $x$ belongs to $S\left(
\mathcal{M}, \tau\right) $ if and only if $x\in S(\mathcal{M})$
and there exists $\lambda>0$ such that $\tau(e^{|x|}(\lambda,
\infty))<\infty$.
Alternatively, an unbounded operator $x$
affiliated with $\mathcal{M}$ is  $\tau$-measurable~(see
\cite{FK}) if and only if
$$\tau\left(e^{|x|}\bigl(n,\infty\bigr)\right)\rightarrow 0,\quad n\to\infty.$$

For convenience of the reader,  we also recall the definition of the measure topology $t_\tau$ on the algebra $S(\cM,\tau)$. For every $\varepsilon,\delta>0,$ we define the set
$$V(\varepsilon,\delta)=
\left\{
x\in S(\mathcal{M},\tau):\ \exists p\in \cP\left(\mathcal{M}\right)\mbox{ such that } \left\|x(\mathbf{1}-p)\right\|_\infty \leq\varepsilon,\ \tau(p)\leq\delta
\right\}.$$
 The topology generated by the sets $V(\varepsilon,\delta)$, $\varepsilon,\delta>0,$ is called the \emph{measure topology} $t_\tau$ on $S(\cM,\tau)$ \cite{DPS, FK, Nelson}.     It is well known that the algebra $S(\cM,\tau)$ equipped with the measure topology is a complete metrizable topological algebra.
  A sequence $\left\{x_n\right\}_{n=1}^\infty\subset S(\cM,\tau)$ converges to zero with respect to measure topology $t_\tau$ if and only if $\tau\big(e^{|x_n|}(\varepsilon,\infty)\big)\to 0$ as $n\to \infty$ for all $\varepsilon>0$ \cite{DPS}.

Another important vector topology on $S(\cM,\tau)$ is the \emph{local measure topology}.
For convenience we denote by $P_f(\cM)$ the collection of all $\tau$-finite projections in $\cM$, that is, the set of all $e\in \cP(\cM)$ satisfying $\tau(e) <\infty$.
A neighbourhood base for this topology is given by the sets $V(\varepsilon, \delta; e)$, $\varepsilon, \delta>0$, $e\in P_f(\cM)$, where $$V(\varepsilon,\delta; e) =
\left\{ x\in S(\cM,\tau)~: ~ exe \in V(\varepsilon, \delta)\right \}.$$
Obviously, local measure topology is weaker than measure topology\cite{DP2}.
We note here, that the local measure topology used in the present paper differs from the local measure topology defined in e.g. \cite{BCS,BCS_2014}.

%
%

\begin{definition}\label{mu}
Let $\mathcal M$ be  a von Neumann  algebra equipped
with a faithful normal semi-finite trace $\tau$ and let $x\in
S(\mathcal{M},\tau)$. The generalized singular value function $\mu(x):t\mapsto  \mu(t;x)$ of
the operator $x$ is defined by setting
$$
\mu(s;x)
=
\inf\{\left\|xp\right\|_\infty:\ p\in \cP(\cM)\mbox{ with}\ \tau(\mathbf{1}-p)\leq s\}, \forall s\in (0,\infty ).
$$
\end{definition}
An equivalent definition in terms of the
distribution function of the operator $|x|$ is the following. for every
operator $x\in S(\mathcal{M},\tau),$ setting
$d_{|x|}(t)=\tau(e^{|x|}(t,\infty)),\quad t>0,$
we have (see e.g. \cite{FK})
\begin{align}\label{dis}
\mu(t; x)=\inf\{s\geq0:\ d_{|x|}(s)\leq t\}.\end{align}

Note that $\mu(x)$ is a function defined on $(0,\infty)$ even if  the trace $\tau$ is finite. In particular,  $\mu(s;x)=0$ when $s\ge \tau({\bf 1})$.


Consider the algebra $\mathcal{M}=L^\infty(0,\infty)$ of all
Lebesgue measurable essentially bounded functions on $(0,\infty)$.
The algebra $\mathcal{M}$ can be seen as an abelian von Neumann
algebra acting via multiplication on the Hilbert space
$\mathcal{H}=L^2(0,\infty)$, with the trace given by integration
with respect to Lebesgue measure $m.$
It is easy to see that the
algebra of all $\tau$-measurable operators
affiliated with $\mathcal{M}$ can be identified with
the subalgebra $S(0,\infty)$ of the algebra of Lebesgue measurable functions $L_0(0,\infty)$ which consists of all functions $x$ such that
$m(\{|x|>s\})$ is finite for some $s>0$.
It should also be pointed out that the
generalized singular value function $\mu(x)$ is precisely the
decreasing rearrangement $\mu(x)$ of the function $|x|$ (see e.g. \cite{KPS}) defined by
$$\mu(t;x)=\inf\{s\geq0:\ m(\{|x|\geq s\})\leq t\}.$$
%
%

If $\mathcal{M}=B(\cH)$ (respectively, $\ell_\infty$) and $\tau$ is the
standard trace ${\rm Tr}$ (respectively, the counting measure on
$\mathbb{N}$), then it is not difficult to see that
$S(\mathcal{M})=S(\mathcal{M},\tau)=\mathcal{M}.$ In this case,
for $x\in S(\mathcal{M},\tau)$ we have
$$\mu(n;x)=\mu(t; x),\quad t\in[n,n+1),\quad  n\geq0.$$
The sequence $\left\{\mu(n;x)\right\}_{n\geq0}$ is just the sequence of singular values of the operator $x.$

If $x,y\in S(\cM,\tau)$, then $x$ is said to be submajorized by $y$, denoted by $x\prec\prec y$, if \begin{align*} \int_{0}^{t} \mu(s;x) ds \le \int_{0}^{t} \mu(s;y) ds  \mbox{ for all $t\ge 0$.} \end{align*}
In particular, for $x,y\in S(0,\infty)$,   $x\prec \prec y$ if and only if  $\int_{0}^{t} \mu(s;x) ds \le \int_{0}^{t} \mu(s;y) ds $, $t\ge 0$.

\subsection{Symmetric spaces}\label{sub:sys}
\begin{definition}\label{def:symmetric}
 A linear subspace $E$ of $S(\cM,\tau)$ equipped with a complete norm $\norm{\cdot}_E$, is called a symmetric space (of $\tau$-measurable operators) if $x\in S(\cM,\tau)$, $y \in E$ and $\mu(x)\le \mu(y)$ imply that $x\in E$ and $\norm{x}_E \le \norm{y}_E$.
\end{definition}

It is well-known that any symmetric space $E$ is a normed $\cM$-bimodule, that is, $axb\in E$ for any $x\in E$, $a,b\in \cM$ and $\left\|axb\right\|_E\leq \|a\|_\infty\left\|b\right\|_\infty \left\|x\right\|_E$ \cite{DP2,DPS}.
 A symmetric space $E(\cM,\tau)\subset S(\cM,\tau)$ is called \emph{strongly symmetric} if its norm $\left\|\cdot\right\|_E$ has the additional property that   $\left\|x\right\|_E \le \left\|y\right\|_E$ whenever $x,y \in E(\cM,\tau)$ satisfy $x\prec\prec y$.
In addition, if $x\in S(\cM,\tau)$, $y \in E(\cM,\tau)$ and $x\prec\prec y$ imply that $x\in E(\cM,\tau)$ and $\left\|x\right\|_E \le \left\|y\right\|_E$, then
 $E(\cM,\tau)$  is called \emph{fully symmetric space} (of $\tau$-measurable operators).

If $E\subset S(\cM,\tau)$ is a symmetric space, then the norm $\left\|\cdot\right\|_E$ is called \emph{order continuous} if $\left\|x_\alpha \right\|_E \rightarrow 0$  whenever $\{x_\alpha\}$ is a downwards directed net in $E^+$ satisfying $x_\alpha \downarrow 0.$
A symmetric space $E(\cM,\tau)$ is said to have the \emph{Fatou property} if for every upwards directed net $\{x_\beta\}$ in $E(\cM,\tau)^+$, satisfying $\sup_\beta \left\|x_\beta\right\|_E <\infty$, there exists an element $x\in E(\cM,\tau)^+$ such that $x_\beta \uparrow x$ in $E(\cM,\tau)$ and $\left\|x\right\|_E = \sup_\beta \left\|x_\beta\right\|_E$.
Examples such as Schatten--von Neumann operator ideals, Lorentz operator ideals, etc. all have symmetric norms which have the  Fatou property\cite{KPS,LT2,DPS}.
If $E$ has the Fatou property and order continuous norm, then it is said to be a \emph{KB}-space (or Kantorovich--Banach space)\cite{DP2}.

If $E(\cM,\tau)$ is a symmetric space, then the carrier projection $c_E\in \cP(\cM)$ is defined by setting
$$c_E = \bigvee\{p:p\in \cP(\cM),~p \in E(\cM,\tau)\}.$$
We remark that, replacing the von Neumann algebra $\cM$ by the reduced von Neumann algebra $\cM_{c_E}$ (note that  $ E(\cM,\tau) = c_E  E(\cM,\tau)  c_E$, see e.g. \cite[Corollary 6]{DP2}), it is often assumed that the carrier projection of $ E(\cM,\tau)$ is equal to ${\bf 1}$.

If $E (\cM,\tau) $ is a symmetric space, then the K\"{o}the dual $E(\cM,\tau)^\times $ of $E(\cM,\tau)$ is defined by
 $$ E(\cM,\tau)^\times =\{   x\in S(\cM,\tau) : \sup_{\|y\|_E\le 1, y\in E}\tau (|xy|)   <\infty    \},$$
and for every $x\in E(\cM,\tau)^\times$, we set
$\left\|x\right\|_{E^\times} = \sup
 \left\{
 \tau(|yx|) : ~y\in E(\cM,\tau), ~\left\|y\right\|_E \le 1\right\}$ (see e.g.
\cite[Section 5.2]{DP2}, see also \cite{LSZ,DDP}).
It is well-known that $\left\|\cdot\right\|_{E^\times}$ is a norm on $E(\cM,\tau)^\times$ if and only if the carrier projection $c_E$ of $E(\cM,\tau)$ is equal to ${\bf 1}$.
In this case, for a  strongly symmetric space
 $E(\cM,\tau)$, the following statements are equivalent  \cite{DP2,DP2012}.
\begin{itemize}
  \item $E(\cM,\tau)$ has the Fatou property.
  \item $E(\cM,\tau)^{\times\times}=E(\cM,\tau)$ and $\left\|x\right\|_E=\left\|x\right\|_{E^{\times\times}}$ for all $X\in E(\cM,\tau)$.
 \item The norm closed unit ball $B_E$ of $E(\cM,\tau)$ is closed in $S(\cM,\tau)$ with respect to the local measure topology .
\end{itemize}

If $E(\cM,\tau)$ is a strongly symmetric space  with $c_E={\bf 1}$ (or a symmetric  space affiliated with a  semifinite  von Neumann algebra which is  either atomless or
atomic with all minimal projections having equal trace), then \cite[Lemma 30]{DP2} (see also \cite[Theorem IV.5.7]{DPS} and \cite{KPS})
\begin{align}\label{bidualequ}
\norm{x}_E=\norm{x}_{E^{\times\times}},~x\in (L_1\cap L_\infty )(\cM,\tau),
\end{align}
and
\begin{align}\label{largerthan1infty}
L_1\cap L_\infty (\cM,\tau)\subset E(\cM,\tau) \subset (L_1+ L_\infty) (\cM,\tau).
\end{align}

A wide class of symmetric operator spaces associated with the von Neumann algebra $\cM$ can be constructed from concrete symmetric function spaces studied extensively in e.g. \cite{KPS}.
Let $(E(0,\infty),\left\|\cdot\right\|_{E(0 ,\infty)})$ be a symmetric function space on the semi-axis $(0,\infty)$.
The pair
$$E(\cM,\tau)=\{x\in S(\cM,\tau):\mu(x)\in E(0,\infty)\},\quad \left\|x\right\|_{E(\cM,\tau)}:=\left\|\mu(x)\right\|_{E(0,\infty)}$$ is a symmetric operator space affiliated with $\cM$ with $c_E ={\bf 1}$ \cite{Kalton_S} (see also \cite{LSZ}).
For convenience, we denote $\left\|\cdot\right\|_{E(\cM,\tau)}$ by $\left\|\cdot\right\|_E$.

Let $E(0,\infty)$ be a symmetric function space. We define a symmetric space \cite[Chapter I.3]{KPS}
$$(L_1+E) (0,\infty):=\{f\in S(0,\infty): \norm{f}_{L_1+E}:= \inf_{x=u+v, u\in L_1(0,\infty),v\in E(0,\infty ) }
\left\{
\norm{u }_1 + \norm{v}_E\right\} <\infty  \}.  $$


\subsection{The ideal  of $\tau$-compact operators}\label{sec:tauCom}

For a self-adjoint operator $x\in S(\cM,\tau)$, we denote by $s(x)$ its support.
The two-sided ideal $\cF(\cM,\tau)$ in $\cM$ consisting of all elements of $\tau$-finite range is defined by setting
$$\cF(\cM,\tau)=\{x\in \cM ~:~ \tau(s(|x|)) <\infty\} .$$

The $C^*$-algebra  $C_0(\cM,\tau) $ of all $\tau$-compact bounded operators can be described as the closure in the norm $\|\cdot\|_\infty$ of the linear span of all $\tau$-finite projections \cite[Definition 2.6.8]{LSZ}.
Equivalently,  $C_0(\cM,\tau) $ the  set  of all elements $x\in \cM$ such that $\tau(E^{|x|}(\lambda,\infty))<\infty$ for every $\lambda >0$ (see e.g. \cite[Chapter II, Section 4]{DPS}).
The space $C_0(\cM,\tau)$ is associated to the ideal of essentially bounded functions vanishing at infinity (see \cite[Lemma 2.6.9]{LSZ}),
 that is,
$$C_0(\cM,\tau)   =  \left\{
 a \in S(\cM,\tau) : \mu(a)\in L_\infty(0,\infty), \ \mu(\infty;a):=\lim_{t\rightarrow \infty}\mu(t;a) =0\right\}.$$
 In particular, if $\tau$ is finite, then $\cM = C_0( \cM,\tau )$ (see e.g. \cite[Page 64]{LSZ}).
The space  $S_0(\cM,\tau)$ of $\tau$-compact operators is the space associated to the algebra of functions from $S(0,\infty)$ vanishing at infinity, that is,
$$S_0(\cM,\tau) = \{x\in S(\cM,\tau) :  \ \mu(\infty;x) =0\}.$$
This is a two-sided ideal in $S(\cM,\tau)$ and, clearly, $C_0(\cM,\tau) = S_0(\cM,\tau)\cap \cM$.

We denote
$$C_0 (0,\infty):=\{f\in L_\infty (0,\infty):  \mu(\infty;f)=0 \},$$
and
$$(L_1 + C_0) (0,\infty):=\{f\in L_1+ L_\infty (0,\infty):  \mu(\infty;f)=0 \}. $$

\section{A noncommutative reflexive gate type theorem}
\label{RGT}

Recall the   reflexive gate type result
for
 symmetric function spaces $E(0,1)$\cite[Corollary 3.2.3]{Novikov}, i.e., if $E(0,1)\ne L_1(0,1) $ and $E(0,1)\ne L_\infty (0,1)$, then there exist two  reflexive symmetric space $F_1(0,1)$ and $ F_2(0,1)$ such that $F_2(0,1)\subset E(0,1)\subset F_1(0,1)$.
For a  symmetric sequence space $\ell_E $ such that $\ell_E \subsetneqq c_0$,
by the well-known  construction of Davis--Figiel--Johnson--Pe{\l}czy\'{n}ski\cite{DFJP}  (see \cite[Theorem 5.37]{AB}, \cite[Chapter VI, Lemma 5.3]{Conway} or \cite[p. 255]{KM} for detailed expositions), there exists a reflexive sequence space $\ell_F\subsetneqq \ell_E$, see e.g. the proof  of \cite[Proposition 3.4]{KM}.
In the following lemma, we show that this reflexive sequence is actually symmetric, which
 is a discrete   analogue of the reflexive gate result\cite[Corollary 3.2.3]{Novikov} for symmetric function spaces on the unit interval.

\begin{lemma}\label{smallref}
Let $\ell_E$ be a symmetric sequence space which does not coincide with
$\ell_1$ (up to equivalent norms).
Then, there exists a reflexive symmetric sequence space $\ell_F$ such that $\ell_F\subset \ell_E$.
\end{lemma}
\begin{proof}The proof relies on the construction of Davis--Figiel--Johnson--Pe{\l}czy\'{n}ski\cite[p.313]{DFJP}.

Let $\left\{e_n\right\}_{n\ge 1}$ be the standard symmetric basis of $\ell_E$.
Since $\ell_E\ne \ell_1$, it follows that  $\left\{e_n\right\}_{n\ge 1}$ is weakly null.
By the Krein--Smulian theorem \cite[Theorem 3.42]{AB},
the convex circled\footnote{A nonempty subset $A$ of a vector space is said to circled (or balanced) whenever $x\in A$ and $0\le \lambda \le 1$ imply $\lambda x\in A)$\cite[p.134]{AB}.
A convex circled hull of $A$ is
$$\left\{ \sum_{i=1}^n \lambda_i x_i :x_i\in A  \mbox{ for each $i$ and } \sum_{i=1}^n |\lambda_i|\le 1\right\}.$$
} hull $W$ of $\left\{e_n\right\}_{n\ge 1}$ is weakly compact.

Let $U_n =2^n W +2^{-n}U$, where $U$ stands for the unit ball of $\ell_E$,  and $\norm{\cdot}_n$ denotes   the Minkowski functional of
$U_n$, i.e.,
$$\norm{x}_n:=\inf\left\{\lambda >0:x\in \lambda U_n \right\}.$$
Put $\ell_F:= \left\{x\in \ell_E :\norm{x}_F: =\left[\sum_{n=1}^\infty \norm{x}_n^2\right]^{1/2}<\infty \right\}$. By \cite[Lemma 1]{DFJP} (see also \cite[Theorem 5.37]{AB}), $\ell_F$ is reflexive.

The ideal property of $\ell_F$ was asserted in \cite[Proposition 3.4]{KM}. For the sake of completeness, we include a short proof below.
Let  $y=\left(y(k)\right)_{k\ge 1}\in \ell_F$ and $x=\left(x(k)\right)_{k\ge 1}\in \ell_\infty$ such that $|x|\le |y|$.
We have $ y\in (\norm{y}_n+\varepsilon)U_n  $ for any $\varepsilon>0$, i.e.,
$$y= \sum_{k=1}^N  s_k e_k
+ y_{n,\varepsilon},   $$
  where $\sum_{k=1}^N |s_{k}|\le 2^n \left(
\norm{y}_n+\varepsilon
\right),~1\le N<\infty $   and $\left( y_{n,\varepsilon }(k)\right)_{k\ge 1}\in \frac{\norm{y}_n+\varepsilon}{2^n}U .$
Note that
$$x=
\sum_{y_k\ne 0}  s_k \frac{x_k}{y_k} e_k
 + \left (\frac{x_k}{y_k}  y_{n,\varepsilon}(k) \right)_{k\ge 1} ,  $$
$\sum_{y_k\ne 0}  \left|s_k\frac{x_k}{y_k}\right|\le  \sum_{y_k\ne 0}   |s_{k}|\le 2^n \left (\norm{y}_n+\varepsilon\right) $ and
$\left(
\frac{x_k}{y_k}  y_{n,\varepsilon}(k)
\right)_{k\ge 1}\in \frac{\norm{y}_n+\varepsilon}{2^n}U $.
Therefore, $\norm{x}_n\le \norm{y}_n$ for every $n\ge 1$, and  $x\in \ell_F$ with
  with $\norm{x}_F\le \norm{y}_F$.

To show that
 $\ell_F$ is symmetric, one only need to observe that for  any permutation $\pi$  and $n\ge 1$, we have
 $\pi(W)=W$ and $\pi(U)=U$.
 Therefore,
$\norm{\left( y(k) \right)_{k\ge 1}}_n =\norm{\left( y\left(\pi(k)\right)\right)_{k\ge 1}  }_n  $ for any $y\in \ell_F$ and  any permutation $\pi$.
%
\end{proof}
\begin{cor}\label{bigref}
Let $\ell_E$ be a symmetric sequence  space which does not coincide with $c_0$ and $\ell_\infty$ (up to equivalent norms).
Then, there exists a reflexive symmetric sequence space $\ell_F$ such that $\ell_F\supset \ell_E$.
\end{cor}
\begin{proof}
Since $\ell_E\ne c_0,\ell_\infty$, it follows that $\ell_E^\times \ne \ell_1$ (see e.g.
\cite[Proposition 11]{HPS} and \cite{KM00}).
By Lemma~\ref{smallref}, we obtain that  there exists a reflexive symmetric sequence space  $\ell_G\subset \ell_E^{\times}$.
Therefore, letting $\ell_F: =\ell_G^\times$, which is a reflexive symmetric sequence space\cite[Section 8.3]{DP2} (see also \cite{DPS,MN,KPS}), we have
$$\ell_F=\ell_G^\times \supset \ell_E^{\times \times }\supset \ell_E,$$
which completes the proof.
\end{proof}

Now, we prove the main theorem of this section, which is an infinite measure version of \cite[Corollary 3.2.3]{Novikov}.
 \begin{theorem}\label{refgate}
 Let $E(0,\infty)$ be a symmetric function space such that
 $$E(0,\infty)\cap L_\infty (0,\infty )\subsetneqq  C_0 (0,\infty). $$
 Then, there exists a symmetric $KB$-function space such that $F(0,\infty)$ containing $E(0,\infty)$.
 If, in addition,  $E(0,1): =\left\{f\in  L_1(0,1): g(x): =\begin{cases}
    f(x), & \text{if $x\in (0,1)$};\\
    0, & \text{otherwise}
  \end{cases}   \in  E(0,\infty) , \norm{f}_{E(0,1)}=\norm{g}_{E(0,\infty)} \right\}$ satisfies that
  $$E(0,1)\ne L_1(0,1), $$
 then  $F(0,\infty)$   can be chosen as a reflexive symmetric space.
 \end{theorem}
\begin{proof}

Let $\ell_E$ be the symmetric sequence space generated by $E(0,\infty)$ by setting
\begin{align}\label{definition:le}\ell_E :=\left
\{(\alpha_n)_{n\ge 1}: \sum_{n\ge 1} \alpha_n \chi_{(n-1,n]} \in E(0,\infty ) \right \}, ~\norm{(\alpha_n)_{n\ge 1}}_{\ell_E}  :=\norm{ \sum_{n\ge 1} \alpha_n \chi_{(n-1,n]} }_{E(0,\infty)}.  \end{align}
By condition $ E(0,\infty)\cap L_\infty (0,\infty ) \subsetneqq  C_0 (0,\infty)$, we obtain  that $\ell_E \subsetneqq c_0$.
Indeed,  $E(0,\infty)\cap L_\infty (0,\infty ) \subsetneqq  C_0 (0,\infty)$ implies that there exists $z=\mu(z)\in C_0(0,\infty)   $ such that
$\mu(z)\notin E(0,\infty)\cap L_\infty (0,\infty )  $, and therefore, by \eqref{definition:le}, we obtain that $\left(\mu(n;z)\right)_{n\ge 1} \in c_0$ but $\left(\mu(n;z)\right)_{n\ge 1} \notin \ell_E $.

By Corollary \ref{bigref}, there exists a reflexive symmetric sequence space
$\ell_G\supset \ell_E$.
Observe that $\ell_G$ is fully symmetric\cite{LT2,DPS,KPS}.

Now, we define a   symmetric function space $F_1(0,\infty)$
by setting
\begin{align}\label{fromsequencetofunction}
F_1(0,\infty):=\left\{
f\in (L_1+L_\infty ) (0,\infty ):   \left(
\int_{n-1 }^{n }\mu(t;f ) dt \right)_{n\ge 1}\in \ell_G \right\}
\end{align}
equipped with the norm\cite[Theorem 3.6.6]{LSZ}
$$ \norm{f}_{F_1(0,\infty )} :=
\norm{
\left(    \int_{n-1}^{n }\mu(t;f )dt    \right)_{n\ge 1}
  }_{\ell_G} , f\in F_1(0,\infty),  $$
which also has the Fatou property and order continuous norm, i.e., $F_1(0,\infty)$ is a symmetric $KB$-space.
Indeed, let $\left\{x_\beta \right\}$ be an upwards directed net in $F_1(0,\infty)^+$ satisfying $\sup_\beta \norm{x_\beta}_{F_1(0,\infty)}<\infty $.
By the definition of $F_1(0,\infty)$,  $\left( \left(    \int_{n-1}^{n }\mu(t;x_\beta )dt    \right)_{n\ge 1}\right)_\beta$
is an upwards directed net in $\ell_G^+$ with $\sup_\beta \norm{\left(    \int_{n-1}^{n }\mu(t;x_\beta )dt    \right)_{n\ge 1} }_{\ell_G}<\infty .$
Let $ x\in S(0,\infty)$ be such that $x_\beta\uparrow x$. Then, we have $\mu(x_\beta)\uparrow x$ and $ \int_{n-1}^{n }\mu(t;x_\beta )dt \uparrow  \int_{n-1}^{n }\mu(t;x )dt $ for each $n\ge 1$. By the Fatou property of $\ell_G$, we obtain that $\left(    \int_{n-1}^{n }\mu(t;x )dt    \right)_{n\ge 1}\in \ell_G$ with   $\norm{\left(    \int_{n-1}^{n }\mu(t;x )dt    \right)_{n\ge 1} }_{\ell_G}=\sup_\beta \norm{\left(    \int_{n-1}^{n }\mu(t;x_\beta )dt    \right)_{n\ge 1} }_{\ell_G} $, i.e., $x\in F_1(0,\infty)$ with $\norm{x}_{F_1(0,\infty)}= \sup_\beta \norm{x_\beta}_{F_1(0,\infty)}=\sup_\beta \norm{x_\beta}_{F_1(0,\infty)}$.
Therefore, $F_1(0,\infty)$ has the Fatou property.
The order continuity of $\norm{\cdot}_{F_1(0,\infty)}$ can be proved by a similar argument.

Note that for any $f=\mu(f)\in E(0,\infty)$, we have
$$
\sum_{n\ge 1} \mu(n; f)\chi_{(n-1,n]}  \le \mu(f) ~\mbox{ and }~ \sum_{n\ge 1} \mu(n; f)\chi_{(n-1,n]} \in E(0,\infty ).
$$
By \eqref{definition:le}, we have
 $\left( \mu(n; f) \right)_{n\ge 1}\in \ell_E$.
 Therefore,
 $$   \left(\int_{n-1 }^{n   }\mu(t;f)dt \right)_{n \ge 2 }=
  \left(\int_{n}^{n+1  }\mu(t;f)dt \right)_{n \ge 1}
    \le
    \left(
 \mu(n;f)
 \right)_{n\ge 1}  \in \ell_E \subset \ell_G.
 $$
Therefore, $\left(
\int_{n-1 }^{n }\mu(t;f)dt \right)_{n\ge 1}\in \ell_G $, i.e., $$ f\stackrel{\eqref{fromsequencetofunction}}{\in} F_1(0,\infty).  $$
That is, $E(0,\infty)\subset F_1(0,\infty)$.
The proof for the  first assertion is complete by defining $F(0,\infty):=F_1(0,\infty)$.

Assume that  $E(0,1)\ne L_1(0,1) $.
If $E(0,1)= L_\infty (0,1)$, then, defining $G(0,1):=L_2(0,1)$ (which is a reflexive symmetric function space), we have    $G(0,1)\supset E(0,1)$.
If $E(0,1)\ne L_\infty (0,1)$, then, by \cite[Corollary 3.2.3]{Novikov},
  there exists a reflexive symmetric function space $G(0,1)$ such that $E(0,1)\subset G(0,1)$.
We  define $$G(0,\infty):=
\left\{
f\in (L_1+L_\infty)(0,\infty): \mu(f)\chi_{(0,1)}\in G(0,1)
\right\} $$
 equipped with the norm\footnote{
 Let $f,g\in G(0,\infty)$.
 By \cite[Corollary 7.6]{Bennett_S}, there exists $A\subset (0,\infty)$ such that
 $\mu( (f+g)\chi_{A} ) = \mu(f+g)\chi_{(0,1)}$.  Therefore, we have $\norm{\mu(f+g)\chi_{(0,1)}}_{G(0,1)} =\norm{(f+g)\chi_{A} }_{G(0,1)}\le \norm{f\chi_{A} }_{G(0,1)}+\norm{g\chi_{A} }_{G(0,1)} \le  \norm{f\chi_{(0,1)} }_{G(0,1)}+\norm{g\chi_{(0,1)} }_{G(0,1)} $.
 Hence, $\norm{\cdot}_{G(0,1)}$ generates a symmetric norm on $G(0,\infty)$. Since $\norm{\cdot}_\infty$ is a symmetric norm, it follows $\norm{\cdot}_G$ is a complete symmetric norm on $G(0,\infty)$, see e.g. \cite[Chapter I.3]{KPS}.
 }
$$\norm{f}_G  :=
  \inf_{f=u+v , ~u\in G(0,\infty ),v\in  L_\infty  (0,\infty)}\left\{
\norm{\mu(u) \chi_{(0,1)} }_{G(0,1)} +\norm{v}_\infty\right\}  ,~f\in G(0,\infty). $$
By the definition of a K\"{o}the dual (see e.g. \cite[(4.31)]{KPS}), it is readily verified that
  the K\"{o}the dual $G(0,\infty)^\times $ of $G(0,\infty)$ is given by
\begin{align*}
&~\quad   \left\{
f\in (L_1+L_\infty)(0,\infty): \norm{\mu(f)\chi_{(0,1)}}_{G(0,1) ^{\times}}<\infty
\right\} \cap L_1(0,\infty )  \\
& =
\left\{
f\in (L_1+L_\infty)(0,\infty): \mu(f)\chi_{(0,1)}\in G(0,1) ^\times  \mbox{ and } \mu(f)\chi_{[1,\infty)}\in L_1  (0,\infty )
\right\}.
\end{align*}
 equipped with the norm
$$ \norm{f}_{G(0,\infty)^\times }:= \max\left\{
 \norm{\mu(f)\chi_{(0,1)}}_{G(0,1) ^{\times}}, \norm{f}_1 \right\},~f\in G(0,\infty)^\times
  .  $$
For any decreasing net $S_0(0,\infty)\supset E(0,\infty)\ni f_i\downarrow 0$, we have  $\mu(f_i)\downarrow_i 0$ (see e.g \cite[Proposition 2(iv)]{DP2}) and therefore, there exists a constant $c(G)$ depending on $G(0,1)$ only\cite[Chapter I, Theorem 1.8]{Bennett_S} such that \begin{align}\label{eqfo}
\norm{f_i}_{G(0,\infty )^\times } \le c(G)\left( \norm{\mu(f_i)\chi_{(0,1)}}_{G(0,1)^{\times}} +\norm{\mu(f_i) }_{L_1(0,\infty )} \right) \downarrow_i 0,
\end{align}
which shows that $G(0,\infty )^\times$ has an  order continuous norm.

Now, we define $$F(0,\infty ) := G(0,\infty )\cap F_1(0,\infty ) ,$$
equipped with $\norm{\cdot}_F:=\max\left\{
\norm{\cdot}_G,\norm{\cdot}_{F_1}
\right\}$.
 Since both $G(0,\infty)$ and $F_1(0,\infty)$ have the Fatou property, it follows that $F(0,\infty ) $ has the Fatou property. The same reasoning as in  \eqref{eqfo} shows that  $\norm{\cdot}_{F}$ is  order continuous.
Observe that\cite[Section 2.3]{SS14}  (see also \cite{Loz})
 $$F(0,\infty)^\times = G(0,\infty)^\times + F_1  (0,\infty)^\times $$
 is  equipped with the norm
  $$ \norm{f}_{F  (0,\infty)^\times} :=
  \inf_{f=u+v , ~u\in G(0,\infty)^\times ,v\in  F_1  (0,\infty)^\times}\left\{
  \norm{ u }_{G ^\times} +\norm{v}_{F_1 ^\times}\right\}.$$
  We claim that $\norm{\cdot}_{F(0,\infty)^{\times}}$
 is order continuous.
 Indeed,  for any decreasing net $F_1(0,\infty)\ni f_i\downarrow 0$, we have $\mu(f_i)\downarrow_i 0$. In particular,  we have
 $\mu(n;f_i)\downarrow _i 0$ for each $n\ge 1$.  By the triangle inequality, we have
 $$\norm{f_i }_{F  (0,\infty)^\times}\le
  \norm{ \mu(f_i)\chi_{(0,1)} }_{G(0,\infty) ^\times} +\norm{ \mu(f_i)\chi_{[1,\infty)}}_{F_1 (0,\infty) ^\times}.$$
By \eqref{eqfo}, it suffices to prove that $\norm{ \mu(f_i)\chi_{[1,\infty)}}_{F_1 (0,\infty) ^\times}\to_i 0$. By the definition of $\norm{\cdot}_{F_1(0,\infty)}$, we have \begin{align*}
 \norm{\mu(f_i)\chi_{[1,\infty )}  }_{F_1(0,\infty)^\times }&
\stackrel{\tiny\mbox{\cite[Prop. 26]{DP2}}}{=}
\sup_{\norm{x}_{F_1(0,\infty )} =1}\int_0^\infty \mu(t; \mu(f_i)\chi_{[1,\infty )} )\mu(t;x)dt\\
&\qquad =\qquad
\sup_{\norm{x}_{F_1(0,\infty )} =1}\sum_{n\ge 1} \int_{n-1}^n \mu(t; \mu(f_i)\chi_{[1,\infty )}  ) \mu(t;x)dt
\end{align*}
and
\begin{align*}  \sum_{n\ge 1} \mu(n +1 ; f_i) \int_{n-1}^n  \mu(t;x)dt& =
 \sum_{n\ge 1} \mu(n ;\mu( f_i)\chi_{[1,\infty ) }) \int_{n-1}^n  \mu(t;x)dt\\
 & \le
 \sum_{n\ge 1} \int_{n-1}^n \mu(t; \mu(f_i)\chi_{[1,\infty )} ) \mu(t;x)dt\\
 &
 \le  \sum_{n\ge1} \mu(n-1 ; \mu(f_i)\chi_{[1,\infty ) } ) \int_{n-1}^n  \mu(t;x)dt=\sum_{n\ge 1} \mu(n   ; f_i) \int_{n-1}^n  \mu(t;x)dt.
\end{align*}
By the definition of $F_1(0,\infty)$, we have  $(\mu(n;f_i))_{n\ge 1}\in \ell_G^\times $ and
$$\norm{\mu(f_i)\chi_{[1,\infty )}  }_{F_1(0,\infty)^\times } \le \norm{\{\mu(n;f_i)\}_{n\ge 1} }_{\ell_G^\times }\downarrow_i 0.$$

By \cite[Theorem 1.c.5]{LT2} (see also \cite[Theorem 2.4.14]{MN} or \cite{KPS,DPS}),  $F(0,\infty)$ is reflexive.
\end{proof}

The following lemma is folklore,  it provides a characterization for symmetric spaces whose `tail' is a proper subspace of $C_0(\cM,\tau)$. For the sake of completeness, we include the full proof.
\begin{lemma} \label{bidualC0}
Assume that
\begin{enumerate}
  \item $E(\cM,\tau)$ is  a strongly symmetric space affiliated with a von Neumann algebra equipped with a semifinite infinite faithful normal trace $\tau$,  with $c_E={\bf 1}$;
  \item $\cM$ is an atomless von Neumann algebra equipped with a  semifinite infinite faithful normal trace $\tau$ or $\cM$ is atomic equipped with a  semifinite infinite faithful normal trace $\tau$ such that  all minimal projections having equal trace.
\end{enumerate}
Then, $E(\cM,\tau)^{\times \times}\subset S_0(\cM,\tau)$
if and only if
$ E(\cM,\tau)\cap \cM  \subsetneqq   C_0 (
 \cM,\tau )  $.
\end{lemma}
\begin{proof}
$(\Rightarrow)$.
Assume, contrapositively, that
$ E(\cM,\tau)\cap \cM $ is not a proper subspace of $C_0(\cM,\tau)$.
Note that, if $E(\cM,\tau)$ contains some element $x$ such that $\mu(\infty;x)>0$, then $E(\cM,\tau)\supset \cM\supset C_0(\cM,\tau)$.
If for every $x\in E(\cM,\tau)$, we have $\mu(\infty;x)=0$, then $E(\cM,\tau)\cap \cM\subset C_0(\cM,\tau)$.
Therefore,
$ E(\cM,\tau)\cap \cM  =  C_0 (
 \cM,\tau )  $ or $ E(\cM,\tau)\cap \cM  =  \cM  $, i.e.,
 $E(\cM,\tau)\supset  E(\cM,\tau)\cap \cM  \supset   C_0 (
 \cM,\tau )  $.

 By \cite[Definition 5.1]{DDP93} 
 (see also \cite[Chapter IV, Proposition 3.12]{DPS}),
  we have
   \begin{align*}
   E(\cM,\tau)^{\times} &= \left\{y\in S(\cM,\tau):xy \in L_1(\cM,\tau)\mbox{ for all }x\in E(\cM,\tau)\right\}\\&\subset
   \left\{y\in S(\cM,\tau):xy \in L_1(\cM,\tau)\mbox{ for all }x\in C_0(\cM,\tau)\right\}\\
   &=C_0(\cM,\tau)^\times
   \stackrel{\mbox{\tiny see e.g. \cite[Lemma 8]{SS14} }}{=}L_1(\cM,\tau),
   \end{align*}
 and
    \begin{align*}
  E(\cM,\tau)^{\times  \times } &= \left\{y\in S(\cM,\tau):xy \in L_1(\cM,\tau)\mbox{ for all }x\in E(\cM,\tau)^{\times }\right\}\\&\supset
   \left\{y\in S(\cM,\tau):xy \in L_1(\cM,\tau)\mbox{ for all }x\in L_1(\cM,\tau)\right\}
   =\cM,
   \end{align*}
 which is a contradiction  with the assumption that $E(\cM,\tau)^{\times \times}\subset S_0(\cM,\tau)$.

$(\Leftarrow)$.
Assume by contradiction that
$E(\cM,\tau)^{\times \times}\supset \cM $.

By \cite[Proposition 28]{DP2} (see also \cite{DDP93,DP2012}), there
exists a fully symmetric function space $G(0,\infty )$ having the Fatou property
such that
$$E (\cM,\tau)^{\times \times }  =G(\cM,\tau)$$
with $\norm{y}_{E (\cM,\tau)^{\times \times} } = \norm{\mu(y)}_{G(0,\infty )}$, $y\in E (\cM,\tau)^{\times \times} $.
Since $G(\cM,\tau)\supset \cM$, it follows that $G(0,\infty )\supset L_\infty(0,\infty )$ and therefore,
   we have
 $\norm{\cdot}_\infty \ge c(G)\norm{\cdot}_{G}$ for some constant $c(G)$ depending on $G(0,\infty)$ only \cite[Chapter I, Theorem 1.8]{Bennett_S}.
 Therefore, we have
  \begin{align*}
  E(\cM,\tau)\cap \cM \stackrel{\eqref{largerthan1infty}}{\supset}\overline{(L_1\cap L_\infty )(\cM,\tau)}^{\norm{\cdot}_{E} } \cap \cM & \stackrel{\eqref{bidualequ}}{=}   \overline{(L_1\cap L_\infty )(\cM,\tau)}^{\norm{\cdot}_{G} } \cap \cM \\
  &\supset    \overline{(L_1\cap L_\infty )(\cM,\tau)}^{\norm{\cdot}_\infty } =   C_0 (\cM,\tau) ,
  \end{align*}
which is a contradiction with the assumption that $ E(\cM,\tau)\cap \cM  \subsetneqq   C_0 (
 \cM,\tau )  $.
\end{proof}

Below,  we establish a noncommutative version of Theorem \ref{refgate}.
Observe that the result of Proposition \ref{prop:ncembe} below
{holds for both finite and infinite traces.}
\begin{proposition}\label{prop:ncembe}
Let $\cM $ be a
von Neumann algebra equipped with a semifinite faithful normal trace $\tau$.
Let $E(\cM,\tau)$ be a strongly symmetric space such that
 $E(\cM,\tau)^{\times \times}\subset S_0(\cM,\tau)$\footnote{Observe that if $\tau({\bf 1})<\infty$, then $S_0(\cM,\tau)=S(\cM,\tau)$.}.
Then, there exists a symmetric $KB$-function space $F(0,\infty)$
such that
$$E(\cM,\tau)\subset F(\cM,\tau).$$ 
\end{proposition}
\begin{proof}
Without loss of generality, we may assume that $c_E ={\bf 1}$.
If $\tau({\bf 1})<\infty$, then the assertion is trivial
as $E(\cM,\tau)\stackrel{\eqref{largerthan1infty}}{\subset} L_1(\cM,\tau)$
and  we can take $F(0,\infty ):=L_1(0,\infty)$.

Now, we assume that $\tau({\bf 1})=\infty $. 

By \cite[Proposition 28]{DP2} (see also \cite{DDP93}), there
exists a fully symmetric function space $G(0,\infty )$ having the Fatou property
such that
$$E (\cM,\tau)^\times =G(\cM,\tau)$$
with $\norm{y}_{E^\times(\cM,\tau)} = \norm{\mu(y)}_{G(0,\infty )}$, $y\in E (\cM,\tau)^\times$.

Let $F_1(0,\infty ):=G(0,\infty)^\times $.
In particular, $F_1(0,\infty)$ has the Fatou property\cite[Theorem 27]{DP2}.
Observe that $F_1(\cM,\tau)= E(\cM,\tau)^{\times\times}$\cite[Theorem 53]{DP2}. We have \cite[Section 5.3]{DP2}
$$E(\cM,\tau)\subset E(\cM,\tau)^{\times \times}=F_1(\cM,\tau) \subset S_0(\cM,\tau) . $$
In particular, $F_1(0,\infty )\subset S_0(0,\infty )$. Since both
$F_1(0,\infty) $ and $L_\infty(0,\infty)$ have the Fatou property, it follows that $F_1(0,\infty)\cap L_\infty (0,\infty )$ has the Fatou property and therefore,  $F_1(0,\infty)\cap L_\infty (0,\infty )\ne  C_0 (0,\infty)$.

By Theorem \ref{refgate}, there exists
  a symmetric $KB$-function space $F(0,\infty)$ such that $F_1(0,\infty)\subset F(0,\infty)$, which completes the proof.
\end{proof}

\section{Main results}
\label{s:M}
The starting point of this section is the following important
result  due to Akemann, Dodds and Gamlen\cite[Theorem 4.2]{ADG} (see also \cite[Corollary II.9 and Theorem IV.3]{Akemann} and  \cite{Sakai}).
 It should be viewed
     as a noncommutative version of Grothendieck's theorem\cite{Grothendieck}, which states that  an arbitrary bounded operator from $C(K)$ into a weakly sequentially complete Banach space is necessarily  weakly compact.

\begin{theorem}\label{ADG}
A bounded linear map from a $C^*$-algebra into a weakly
sequentially complete Banach space is weakly compact.
\end{theorem}

Recall that a noncommutative strongly symmetric space is a \emph{KB}-space if and only if it is weakly
sequentially complete\cite[Theorem 6.5]{DPS16}, see also \cite{DP2,DPS,DP2012,DDP}.
We have the   following consequence of Theorem \ref{ADG}.
\begin{corollary}\label{p1}
Let $\cM$ be a von Neumann algebra equipped with a semifinite
faithful normal trace $\tau$, and let
$E(\cM,\tau)$ be a strongly symmetric $KB$-space affiliated with
$\cM$.
Then
every bounded map from a $C^*$-algebra into $E(\cM,\tau)$
is weakly compact.
\end{corollary}

Proposition \ref{proprwc} below
is a consequence of
  Corollary  \ref{p1} and \cite[Corollary 2.9]{DDP}
 (see also \cite[Theorem 5.4]{DSS}).
 For a Banach space $X$, we denote by $B_X$ the unit ball of $X$.
\begin{proposition}\label{proprwc}Let $\cM$ be a von Neumann algebra equipped with a semifinite
faithful normal trace $\tau$.
Assume that  $E(\cM,\tau )$ is a strongly symmetric KB-space such that $E(\cM,\tau)^\times \subset S_0(\cM,\tau)$.
Let $\cA$ be a $C^*$-subalgebra  $  \cM$ and let
$T$ be a bounded linear  operator from $\cA$ into $E(\cM,\tau)$.
Then,  the set
$$ B_\cM T(B_\cA)B_\cM :=\left\{aT(x)b :  a,b\in B_\cM  ,x\in B_\cA   \right \}$$
is relatively weakly compact in  $E(\cM,\tau )$.
\end{proposition}
 \begin{proof}
Without loss of generality, we may assume that $c_E ={\bf 1}$.


 Since every weakly compact operator sends bounded sets into weakly compact ones, we infer from Corollary \ref{p1} that the set
 $$  T(B_\cA) $$
is relatively weakly compact in $E(\cM,\tau)$.
 Since $E(\cM,\tau)$ has the Fatou property, it follows that $E(\cM,\tau)=E(\cM,\tau)^{\times \times}\subset S_0(\cM,\tau)$ (see Section \ref{sub:sys}).
Since  $
E(\cM,\tau)^\times  \subset S_0(\cM,\tau)
$,
it follows from   \cite[Corollary 2.9]{DDP}
  (or to \cite[Theorem 5.4]{DSS})   that
  $$\bigcup_{y\in B_\cA}   \Omega(T(y)) $$ is   relatively $\sigma(E^{\times \times},E^\times)$-compact (equivalently, relatively $\sigma(E,E^\times)$-compact, or weakly compact)  in $E(\cM,\tau)$,
  where $\Omega(x):=\{z\in (L_1+L_\infty) (\cM,\tau): z\prec \prec y\}$.
  Since
  $B_\cM  T(B_\cA)B_\cM  \subset \bigcup_{y\in B_\cA}   \Omega(T(y)) $, it follows that
$$B_\cM  T(B_\cA)B_\cM  $$
   is weakly compact  in $E(\cM,\tau)$
 \end{proof}
\begin{remark}\label{charbidual}
Assume that  $E(\cM,\tau )$ is a strongly symmetric space with $c_E ={\bf 1}$.
Note that, if $\tau({\bf 1})<\infty$, then the condition $E(\cM,\tau)^\times \subset S_0(\cM,\tau)$ holds for any symmetric space $E(\cM,\tau)$.

If $\tau({\bf 1})=\infty$, then $E(\cM,\tau)^\times \subset S_0(\cM,\tau)$ if and only if $E(\cM,\tau)\cap \cM \supsetneqq L_1(\cM,\tau)\cap \cM$.
Indeed, assume that
   $E (\cM,\tau) \cap \cM \supsetneqq  L_1 (\cM,\tau)\cap \cM  $. Then, there exists an element $0\le z\in E(\cM,\tau)\cap \cM $ but $z\notin L_1(\cM,\tau)$. In particular, we have $\tau(z)=\infty$.
 By the definition of K\"{o}the duals,  we infer that ${\bf 1}\notin E(\cM,\tau)^\times$, which, in turn,  implies that $E(\cM,\tau)^\times \subset S_0(\cM,\tau)$.
  On the other hand, assume by contradiction that
$L_1(\cM,\tau)\cap \cM$  is not a proper subspace of $E(\cM,\tau)\cap \cM $. By \eqref{largerthan1infty}, we have $L_1(\cM,\tau)\cap \cM\subset E(\cM,\tau)\cap \cM $. Therefore,
we obtain that
$E(\cM,\tau)\cap \cM =L_1(\cM,\tau)\cap \cM $.
By the fact that    $E(\cM,\tau)\stackrel{\eqref{largerthan1infty}}{\subset} (L_1+L_\infty)(\cM,\tau)$, we obtain that
for any element $x\in E(\cM,\tau)$, $\mu(x)\chi_{(0,1)}\in L_1(0,1)$.
Hence,
all elements in $E(\cM,\tau)$
belong to $L_1(\cM,\tau)$. Therefore, $E(\cM,\tau)^\times \supset L_1(\cM,\tau)^\times =\cM$ \cite[Definition 5.1]{DDP93}. That is, $E(\cM,\tau)^\times \not\subset S_0(\cM,\tau)$, which completes the proof.

\end{remark}

Let $\cU(\cA)$ denote the set of all unitary elements  in a $C^*$-algebra $\cA$.
\begin{proposition}\label{prop:rwc}Let $\cM$ be a von Neumann algebra equipped with a semifinite
faithful normal trace $\tau$ and let $\cA$ be a  unital $C^*$-subalgebra  $  \cM$.
Assume that  $E(\cM,\tau )$ is a strongly symmetric KB-space such that
$E(\cM,\tau)^\times \subset S_0(\cM,\tau)$.
Let $\delta:\cA\to E(\cM,\tau)$ be a derivation.
Then,
$$\{   \delta(u)u^* \mid u\in \cU(\cA) \}$$
is relatively weakly compact  in  $E(\cM,\tau )$.
Consequently,   the
closure $\overline{\conv\{   \delta(u)u^* \mid u\in \cU(\cA) \}}^{\norm{\cdot}_{E} }$ of the convex hull  is  weak compact.
\end{proposition}
\begin{proof}
By Ringrose's theorem\cite[Theorem 2]{Ringrose}, $\delta$ is bounded from $\left(\cA, \norm{\cdot}_\infty \right)$ into $\left( E(\cM,\tau),\norm{\cdot}_E\right)$.
By Proposition~\ref{proprwc}, $\{   \delta(u)u^* \mid u\in \cU(\cA) \}$ is relatively weakly compact in  $E(\cM,\tau )$.

The second assertion follows from the Krein--Smulian theorem (see e.g. \cite{Whitley}).
\end{proof}
The following lemma shows that derivations into a ``large'' symmetric space are  inner.
 \begin{lemma}\label{theorem:L_1+E}
Let $\cM$ be a   von Neumann algebra with a faithful normal semifinite trace $\tau$ and let $\cA$ be a unital $C^*$-subalgebra of $\cM$.
Let $E(\cM,\tau)$ a strongly symmetric   KB-space such that
$E(\cM,\tau)^\times \subset S_0(\cM,\tau)$.
For every derivation $\delta:\cA \rightarrow E(\cM,\tau)$, there exists an element
 $a \in \overline{\conv\{\delta(u)u^*  \mid u\in \cU(A)\}}^{\left\|\cdot\right\|_{E }} $ such that $\delta=\delta_a$ on $\cA$.
In particular, $\left\|a\right\|  \le \left\|\delta\right\|_{\cA \rightarrow E}$.
\end{lemma}
\begin{proof}
Without loss of generality, we  may assume that the carrier projection $c_E ={\bf 1}$.

For every $u\in \mathcal{U}(\cA)$, we have $\delta(u)\in  E(\cM,\tau)$, and therefore we can define the mapping $\alpha_u:  E(\mathcal{M},\tau)\longrightarrow   E(\mathcal{M},\tau)$, by setting
$$\alpha_u(x):=uxu^*+\delta(u)u^*.$$
For every $u,v\in \cU(\cA)$, we have
\begin{align*}
\alpha_{u}(\alpha_{v}(x))
&=uvxv^*u^*+u\delta(v)v^*u^*+\delta(u)u^*\\
&= (uv)x(uv)^*+u\delta(v)v^*u^*+\delta(u)vv^*u^*\\
&=(uv)x(uv)^*+\delta(uv)(uv)^*=\alpha_{uv}(x).
\end{align*}
In addition, the equality $\delta(\mathbf{1})=\delta(\mathbf{1}^2)=2\delta(\mathbf{1})$ implies that $\delta(\mathbf{1})=0$, and therefore $\alpha_{\mathbf{1}}(x)=x, \, x\in   E(\mathcal{M},\tau)$. Thus, $\alpha$ is an action of the group $\mathcal{U}(\cA)$ on $ E (\cM,\tau)$.

We claim that the set $$K:=\overline{\conv \left\{\delta(u)u^*  \mid u\in \cU(A)
\right\}}^{
\left\|\cdot\right\|_{E }} $$ is invariant  with respect to $\alpha$.
Since $\delta(u)u^*=\alpha_u(0)$, it follows that $k_{00}:= \left\{\delta(u)u^*  \mid u\in \cU(A)\right\}$ is an orbit of $0$ with respect to $\alpha$, and therefore, is an invariant subset with respect to $\alpha$.
In addition, for any positive scalars $s$ and $t$ with $\ s+t=1$, we have
\begin{align*}
\alpha_u(s\cdot x+t\cdot y)
 =s\cdot uxu^*+t\cdot uyu^*+(s+t)\cdot \delta(u)u^* =s\cdot \alpha_u(x)+t\cdot \alpha_u(y),\quad \forall  x,y\in   E (\cM,\tau).
\end{align*}
Hence, for every $u\in\cU(\cA)$ the mapping $\alpha_u$ is affine, which implies that $\conv(K_{00})$ is also an invariant subset with respect to $\alpha$.
Now, the equality $\alpha_u(x)-\alpha_u(y)=u(x-y)u^*,\, x,y\in E(\cM,\tau)$ implies that every  $\alpha_u, u\in \cU(\cA),$ is an isometry on $E (\cM,\tau)$.
 Hence, $K$ is an invariant subset with respect to $\alpha$.

Furthermore, the fact that $\alpha_U$ is an isometry combined with \cite[Chapter V, Lemma 10.7]{Conway} implies  that  the family $\left\{\alpha_u:\ u\in\mathcal{U}(\cA)\right\}$ is a noncontracting family of affine mappings.
Clearly,  $\alpha_u$ is weakly continuous for every $u\in \cU(\cA)$.

On the other hand, by Proposition \ref{prop:rwc},
$K$ is weakly compact.
Thus, the set $K$ and the family $\left\{\alpha_u:\ u\in\mathcal{U}(\cA)\right\}$ satisfy the assumptions of the Ryll-Nardzewski fixed-point theorem \cite[Chapter V, Theorem 10.8]{Conway}. Hence, there exists a point $a \in K$ fixed with respect to  $\alpha$, that is, we have $$a=\alpha_u(a)=uau^*+\delta(u)u^*$$ for every  $u\in\cU(\cA)$  and $\left\|a\right\|_{  E}\leq \left\|
\delta
\right\|_{\cA\to  E}$. Therefore $au=ua+\delta(u)$ for every $u\in \cU(\cA)$.
Thus, $\delta(u)=[a,u]$ for every $ u \in\mathcal{U}(\cA)$.
Since every element $x\in\cA$ is a linear combination of four elements from $\mathcal{U}(\cA)$\cite[Theorem 4.1.7]{KR1}, we obtain that $\delta= \delta_{a}$ on $\cA$.
\end{proof}

 A symmetric function  space $E(\cM,\tau )$ is said to have the Levi property\cite[Definition 7]{AA}, or a monotone complete norm,
or to satisfy property (B)\cite[Chapter X.4]{KA}, if
for every upwards directed net $\{x_\beta\}$ in $E(\cM,\tau )^+$, satisfying $\sup_\beta \left\|x_\beta\right\|_E <\infty$, there exists an element $x\in E(\cM,\tau)^+$ such that $x_\beta \uparrow x$ in $E(\cM,\tau )$.
It is well known  that if a norm is Levi, then necessarily it is also weak
Fatou\cite[p.89]{AA}, i.e., there exists a constant $K\ge 1$ such that
$$0\le x_\beta \uparrow x \Rightarrow \norm{x}_E \le K\lim_\beta \norm{x_\beta }_E.$$
\begin{remark}\label{Levi}

For a symmetric function space $E(0,\infty)$, it has the Levi property if and only if it has an equivalent symmetric norm having the Fatou property.
The $(\Leftarrow)$ implication is clear.
For the $(\Rightarrow )$ implication, one only need to observe that
for any $0\le x\in E(0,\infty )$, there exists a net $(x_\beta)_\beta  \subset  (L_1\cap L_\infty)(0,\infty )$ such that $x_\beta   \uparrow x$\cite[Proposition 1(vii)]{DP2}, and therefore,
$$\frac{1}{K} \norm{x}_E \le \sup _\beta \norm{x_\beta}_E \stackrel{\eqref{bidualequ}}{=}\sup _\beta \norm{x_\beta}_{E^{\times \times }} = \norm{x}_{E^{\times \times }} \stackrel{\tiny \mbox{\rm \cite[Section 5.3]{DP2}}}{\le} \norm{x}_{E},$$
and  $\left( E(0,\infty)^{\times \times },\norm{\cdot}_{E^{\times\times}} \right) $ has the Fatou property (see \cite[Theorem 4.1]{DDST} and  \cite{KPS}).

Recall that  $E(0,\infty)$ has the Fatou property if and only $E(0,\infty)$ is isometric to $E(0,\infty)^{\times\times}$.
We obtain that
$E(0,\infty)$ has the Levi property if and only $E(0,\infty)=E(0,\infty)^{\times\times}$ (up to equivalent norms).
\end{remark}

The following theorem is the main result of the present paper, which
resolves
 problems considered in a numbers of papers,  see e.g. \cite{Kaftal_W,Ber_S,HLS,Weigt,Johnson,Hoover,BHLS,BHLS2} and references therein.
\begin{theorem}\label{main}
Let $\cM$ be a von Neumann algebra with a faithful normal semifinite  trace $\tau$ and let $\cA$ be a  $C^*$-subalgebra of $\cM$.
If $E(\cM,\tau )$ is a  strongly  symmetric space of $\tau$-compact operators (i.e., $E(\cM,\tau )\subset S_0(\cM,\tau)$\footnote{If $\tau({\bf 1})<\infty$, then $E(\cM,\tau)\subset S_0(\cM,\tau)$
holds for any symmetric space $E(\cM,\tau)$ affiliated with $\cM$. }) having    the Fatou property (resp., the Levi property), then every derivation $\delta:\cA \rightarrow E(\cM,\tau)$ is inner.
That is, there exists an element $a \in  \overline{\conv
\left\{\delta(u)u^*  \mid u\in \cU(\cA)\right\}}^{t_\tau} \subset E(\cM,\tau)$  with $\|a \|_E\le \|\delta\|_{\cA\rightarrow E}$  ($\left\|a \right\|_E\le c\left\|\delta\right \|_{\cA\rightarrow E}$ for some constant $c$ depending on $E(0,\infty)$ only)  such that $\delta=\delta_a $ on $\cA$.
\end{theorem}
\begin{proof}
Without loss of generality, we may assume that the carrier projection $c_E ={\bf 1}$ and assume that $\cA$ is unital (see e.g. the proof of \cite[Proposition 3.4]{BHLS}).
By Remark \ref{Levi}, we may assume that $E(\cM,\tau)$ has the Fatou property.

Since $E(\cM,\tau)$ has the Fatou property and $E(\cM,\tau)\subset S_0(\cM,\tau)$, it follows that $E(\cM,\tau)^{\times \times }\subset S_0(\cM,\tau)$.
By Proposition \ref{prop:ncembe}, there exists a symmetric $KB$-function space $F(0,\infty)\subsetneqq (L_1+C_0)(0,\infty) $
such that
$$ E(\cM,\tau)\subset F(\cM,\tau)\subset S_0(\cM,\tau). $$
Without loss of generality, we may assume that $L_2(0,\infty )\subset F(0,\infty)$ by replacing $F(0,\infty)$ with $L_2(0,\infty )+F(0,\infty)$.
In particular, $F(0,\infty)\cap L_\infty (0,\infty )\supsetneqq L_1(0,\infty)\cap L_\infty (0,\infty )$. By Remark \ref{charbidual}, $F(0,\infty)^{\times }\subset S_0(0,\infty)$, and therefore, $F(\cM,\tau)^{\times }\subset S_0(\cM,\tau)$.

%
Note that $\delta(\cA)\subset E(\cM,\tau)\subset F(\cM,\tau)$.
By   Lemma  \ref{theorem:L_1+E}, there is an element  $$a\in \overline{\conv
\left\{\delta(u)u^*  \mid u\in \cU(\cA)\right\}}^{\norm{\cdot}_F} $$ such that $\delta=\delta_a$ on $\cA$.
Hence,  there exists a sequence $$\left(x_n\right)_{n=1}^\infty\subset \conv \left\{\delta(u)u^*  \mid u\in \cU(\cA)\right\}$$ such that $\left\|x_n -a\right\|_F\rightarrow_n 0$.
  Since $F(\cM,\tau)$ is a symmetric space, it follows from  \cite[Proposition 20]{DP2} that  $x_n\rightarrow_{t_\tau} a $ as $n\to \infty$.

By Ringrose's theorem \cite{Ringrose}, we have that $\delta: (\cA,\norm{\cdot}_\infty) \rightarrow \left(E(\cM,\tau),\norm{\cdot}_E\right)$ is a bounded mapping.
Since $E(\cM,\tau)$ has the Fatou  property, it follows   that the closed ball $\left(E(\cM,\tau),\norm{\cdot}_E\right)$ with radius $\left\| \delta\right\|_{\cA\rightarrow E}$ is closed in $S(\cM,\tau)$ with respect to the local measure topology  (see Section \ref{sub:sys}).
Noticing that every element $x_n$, $n\ge 1$ belongs to the ball of radius
$\left\|\delta\right\|_{\cA \rightarrow E}$ in $E(\cM,\tau)$ and $x_n\rightarrow a$ in the local measure topology,
we conclude that $a\in E(\cM,\tau)$ with $\left\|a\right\|_E\le \left\|\delta\right\|_{\cA\rightarrow E}$.
\end{proof}

Recall that a    symmetric function  space having  the Fatou/Levi property generates a noncommutative strongly symmetric space having the Fatou/Levi property (see e.g. \cite[Proposition 2.a.8]{LT2} and \cite[Theorem 54]{DP2}).
We have the following consequence of Theorem  \ref{main}.
\begin{cor}\label{cormain1}
Assume that
 $E(0,\infty )\subset S_0(0,\infty )$ is a    symmetric function  space having  the Fatou property (resp., the Levi property).
 Then,
a  derivation $$\delta:\cA \rightarrow E(\cM,\tau)$$ is necessarily inner
  for any
 von Neumann algebra $\cM$ with a semifinite  faithful normal trace $\tau$ and  a   $C^*$-subalgebra $\cA$ of $\cM$.
  That is, there exists an element  $a \in \overline{\conv
\left\{\delta(u)u^*  \mid u\in \cU(\cA)\right\}}^{t_\tau} \subset E(\cM,\tau)$  with $\left\|a \right\|_E\le \left\|\delta\right \|_{\cA\rightarrow E}$ ($\left\|a \right\|_E\le c\left\|\delta\right \|_{\cA\rightarrow E}$ for some constant $c$ depending on $E(0,\infty)$ only) such that $\delta=\delta_a $ on $\cA$.
\end{cor}

Corollary \ref{cormain1}  provides an alternative proof   for the resolution of the derivation problem due to  Bunce and Paschke\cite{BP},
 which differs from those in
  \cite{BGM,Pfitzner}   involving $L$-embeddedness of the predual of a von Neumann algebra.
 Moreover, we provide new information that implementing element $a$ of the derivation can be found in the measure (or any symmetric norm topology which is   weaker  than the $L_1$-topology) closure of the convex hull of $ \{\delta(u)u^*  \mid u\in \cU(A)\}$, which seems to be more natural in   fixed point
theory and derivation theory.
 \begin{cor}
 Let $\cM$ be a semifinite von Neumann algebra and let $\cA$ be a $C^*$-subalgebra of $\cM$. Then,  any  derivation from $\cA$ into $\cM_*$ is inner, and there exists an element
 $a \in  \overline{\conv\{\delta(u)u^*  \mid u\in \cU(A)\}}^{t_\tau}$
   with $\left\|a \right\|_{L_1}\le \left \|\delta\right\|_{\cA\rightarrow L_1}$ such that $\delta=\delta_a $ on $\cA$.

 \end{cor}

The following corollary significantly extends  results from \cite{BHLS,Kaftal_W,Hoover} for derivations into  symmetric ideals.
On the other hand,
it shows that \cite[Theorem 1.3]{BHLS2}
holds true even for $C^*$-subalgebras
rather than von Neumann subalgebras.  

 \begin{cor}

 Let $\cM$ be a semifinite von Neumann algebra equipped with a semifinite faithful normal trace $\tau$ and let $\cA$ be a $C^*$-subalgebra of $\cM$.
 Let $E(\cM,\tau)$ be an ideal of $C_0(\cM,\tau)$ having strongly symmetric norm and the Levi property.
 Then, every derivation from $\cA$ into $E(\cM,\tau)$ is inner.
 \end{cor}

For a von Neumann algebra $\cM$ equipped with a faithful normal tracial state, $\cM\subset S_0(\cM,\tau)$ and $\cM$ has the   Fatou/Levi property. Therefore, we have the following consequence of Theorem~\ref{main}.
 \begin{cor}\cite[Section 5]{Christensen}\cite{Sinclair_S}
 Let $\cM$ be a von Neumann algebra equipped with a faithful normal tracial state and let $\cA$ be a $C^*$-subalgebra of $\cM$.
 Then every derivation from $\cA$ into $\cM$ is inner.
 \end{cor}
 \section{Outer derivations}\label{iffsection}
 In this section, we present several examples of outer derivations with values in a  symmetric space  failing the Levi property.

The study of derivations into ideals of a von Neumann algebra was   initiated Johnson and Parrott\cite{Johnson_P} by showing that
  derivations from an abelian/properly infinite von Neumann subalgebra of $B(\cH)$ into the  algebra $K(\cH)$ of all compact operators on $\cH$ are inner.
  This result may be   well seen as a precursor to Question  \ref{que:1}  in the setting of symmetrically normed ideals in $\cM=B(\cH)$.
  After that, the innerness of derivations into ideals of a von Neumann algebra have been  studied extensively in  \cite{Ber_S_1,Ber_S_2,BHLS,Hoover,BHLS2,Christensen87,Galatan_P,Kaftal_W,Popa_R}.

 Recall that
   the Fatou/Levi property can not be dropped in  Theorem \ref{main}, which has been
 demonstrated in \cite[Example 4.1.8]{Sakai98}, \cite[Theorem 3.8]{BHLS} and \cite[Theorem 4.2.1]{Huang}.
 The following theorem
 demonstrates that
 there exist outer derivations into $E(\cM,\tau)\cap \cM$ whenever  the ideal
   $E(\cM,\tau)\cap\cM$ does not have   the Levi property.

 \begin{theorem}\label{noFatou}
 Let $\cM$ be a non-finite factor equipped with a semifinite faithful normal trace $\tau $and let $E(\cM,\tau)$ be a   symmetric space affiliated with $\cM$ such that  $E(\cM,\tau)\cap \cM $ does not have the Levi property.
 Then, there exist outer derivations from the $C^*$-algebra  $C_0(\cM,\tau)$ into
 $E(\cM,\tau)\cap \cM$.
 \end{theorem}
 \begin{proof}
 Since $\cM$ is a factor, it follows that it is either atomless or
is atomic with all atoms having equal trace.
Therefore,
$$E(\cM,\tau)^{\times\times}$$
is a fully symmetric space having the Fatou property (see e.g. \cite{DDP93}, \cite[Chapter IV, Theorem 5.4]{DPS},  \cite[Remarks 5 and 6]{DP2}).
Hence,
 $$ E(\cM,\tau)^{\times\times }\cap \cM,$$
 equipped with the norm  $\norm{x}_{E(\cM,\tau)^{\times\times }\cap \cM} =\max\left\{ \norm{x}_{E^{\times\times}},\norm{x}_\infty \right\}, ~x\in E(\cM,\tau)^{\times\times }\cap \cM$,
 has the Fatou property.
 On the other hand,
 the lack of the Levi property of $ E(\cM,\tau)\cap \cM $  implies that
%
%
there exists
 a net $(x_i)_{i \in I }$ of positive elements  in $E(\cM,\tau)\cap \cM$ increasing to $0\le x\in \cM$
such that  $\sup\norm{x_i}_{E}<\infty$ but $x\notin E(\cM,\tau)$.
By the Fatou property of  $E(\cM,\tau)^{\times\times}$ and $E(\cM,\tau)\subset E(\cM,\tau)^{\times\times}$\cite[Section 5.3]{DP2},
 we have  $x\in E(\cM,\tau)^{\times \times }\cap \cM$.

 Since $E(\cM,\tau)\cap \cM$ does not have the Levi property, it follows that $E(\cM,\tau)\cap \cM\ne \cM$.
Hence, $E(\cM,\tau)\subset S_0(\cM,\tau)$.
There are two possible cases:

 (1) If $E(\cM,\tau)^{\times\times}\subset  C_0(\cM,\tau)$, then we consider $\delta_x$.

We claim that $\delta_x$ is a non-inner derivation from $ C_0(\cM,\tau) $ into $  E (\cM,\tau)\cap \cM$.
Recall that
$$ \norm{ z}_E\stackrel{\eqref{bidualequ}}{=} \norm{z }_{E^{\times\times }}, ~\forall z\in \cF(\cM,\tau). $$
For every $y\in C_0(\cM,\tau) $, the projections  $e^{|y|}(\varepsilon,\infty)$ and $ e^{|y|}(\varepsilon_1 ,\infty)$ are $\tau$-finite for every $\varepsilon> \varepsilon_1>0$.
Thus, we have
$$\left\|xy e^{|y|}(\varepsilon,\infty)  -xy e^{|y|}(\varepsilon_1 ,\infty)  \right \|_  E \le \left\|
xy e^{|y|}( \varepsilon_1, \varepsilon] \right\|_  E=\left\|
xy e^{|y|}( \varepsilon_1, \varepsilon] \right\|_  {E^{\times\times}} \le \varepsilon \left\|
x \right\|_  {E^{\times\times}} , $$
which shows that $\left(xy e^{|y|}(\frac1n,\infty)\right)_{n\ge 1}$ is a Cauchy sequence in $E(\cM,\tau)$.
On the other hand, $xy e^{|y|}(\frac1n,\infty)\to xy$ in measure, which implies that    $xy\in   E(\cM,\tau)$.
Similarly, $yx\in   E(\cM,\tau)$ and therefore $\delta_x(C_0(\cM,\tau))\subset   E(\cM,\tau)$.
Moreover, $x \in \cM$ and $C_0(\cM,\tau)\subset \cM$ imply that $$\delta_x(C_0(\cM,\tau))\subset E(\cM,\tau)\cap \cM .$$
Finally, if there exists an operator $x'\in   E(\cM,\tau)\cap \cM$ such that $\delta_{x}=\delta _{x'}$, then
$$x-x'\mbox{ commutes with all elements in }C_0(\cM,\tau).$$
Noticing that $\cM$ is the closure of $C_0(\cM,\tau)$ in the  weak operator topology (see e.g. \cite[Definition 2.6.8]{LSZ}), we obtain that
$$x-x'\mbox{ commutes with all elements in } \cM .$$
However, $\cM$ is a factor and therefore  $x-x' \in \mathbb{C}\textbf{1}$, which is a contradiction with the assumption that $E(\cM,\tau)^{\times \times }\subset C_0(\cM,\tau)$.

 (2) If $E(\cM,\tau)^{\times \times}\supset \cM$,  
then, by Lemma \ref{bidualC0}, we have
 $E(\cM,\tau)\cap \cM \supset C_0(\cM,\tau)$.
 Recall that  $E(\cM,\tau)\cap \cM\ne \cM$. We have
  $E(\cM,\tau)\cap \cM = C_0(\cM,\tau)$.
Now, the existence of outer derivations  follows from \cite[Theorem 3.8]{BHLS} immediately.
%
 \end{proof}

 \begin{cor}\label{out bounded}Let $E(0,\infty)$ be a symmetric function space.
 If $E(0,\infty)\cap L_\infty (0,\infty)$ does not have  the Levi property, then for any infinite  factor $\cM$ equipped with a semifinite faithful normal trace $\tau$, there exists outer derivations from the $C^*$-algebra $C_0(\cM,\tau)$ into $E(\cM,\tau) \cap \cM$.
 \end{cor}
 \begin{proof}

By  Theorem \ref{noFatou}, it suffices to prove that $E(\cM,\tau) \cap \cM$ does not have the Levi property.

 Since $E(0,\infty)\cap L_\infty (0,\infty)$ has no Levi  property, it follows that (see e.g. Remark \ref{Levi})
$$E(0,\infty)\cap L_\infty (0,\infty)\ne \left(E(0,\infty)\cap L_\infty (0,\infty)\right)^{\times \times}=  E(0,\infty)^{\times \times }\cap L_\infty (0,\infty).$$
That is, there exists $0\le x=\mu(x)\in E(0,\infty)^{\times \times }\cap L_\infty (0,\infty)$ such that $x\notin E(0,\infty)\cap L_\infty (0,\infty)$.
Observe that
 $$\sum_{n\ge 1}\mu(n;x)\chi_{(n-1,n]} \le \mu(x) \in E(0,\infty)^{\times \times}\cap L_\infty (0,\infty) .$$
There exists $ y\in \cM$ such that $$ \mu(y)=\sum_{n\ge 1}\mu(n;x)\chi_{(n-1,n]} \le \mu(x) \in E(0,\infty)^{\times \times}\cap L_\infty (0,\infty) .$$
 However,
we have
 $$\mu(y)\ge \mu \left(
     \mu(x)  -  \mu(x)\chi_{(0,1]}
     \right) . $$
 Since $\mu(x)\chi_{(0,1]}\in E(0,\infty)\cap L_\infty(0,\infty)$ (see e.g. \cite[p.245]{DP2} or \cite{KPS}), it follows  that
  $ y\notin    E(\cM,\tau)\cap \cM $.
This shows that
 $E(\cM,\tau) \cap \cM$ does not have the Levi property   when $\cM$ a semifinite infinite (atomic or atomless) factor, see Remark \ref{Levi}.
 \end{proof}

Recall that the dilation operator $\sigma_s$, $s>0$, on $S(0,\infty)$ is defined by\cite{KPS,LSZ,LT2}
$$(\sigma_s) f(t) = f\left(\frac{t}{s}\right), \forall f\in S(0,\infty), ~s\in (0,\infty). $$

 For the case of  unbounded  operators,  the example of an outer derivation is more involved.
 The result presented below is perhaps the first example of an outer  derivation into a symmetric space affiliated with a finite von Neumann algebra.
 \begin{example}\label{out01}
 Let $\cM$ be the von Neumann algebra which is the closure of $\oplus_{n\ge 1} \mathbb{M}_{2}$ in the weak operator topology equipped with the trace $\tau:=\oplus_{n\ge 1} \frac{1}{2^{n+1}}{\rm Tr}_2 $, where ${\rm Tr}_2$ is the standard trace of $ \mathbb{M}_{2 }$.
 Let $E(0,1)$
be a symmetric function space without the Levi property.
 Then,  there exist  outer derivations from the $C^*$-algebra $\cA:=\overline{\oplus \mathbb{M}_{2^n}}^{\norm{\cdot}_\infty }$ into $E(\cM,\tau)$.
 \end{example}
 \begin{proof}
The lack of the Levi property of
$E(0,1)$ shows that $E(0,1)\ne E(0,1)^{\times \times }$ (see e.g. Remark~\ref{Levi}),
and
there exists a positive function $a\in E(0,1)^{\times\times}$
but $a\notin E(0,1)$.
Define
$$b:=
\sum_{n\ge 1}\mu\left(\frac{1}{2^{n-1}};a \right)\chi_{[\frac{1}{2^n}, \frac{1}{2^{n-1}})}. $$
We have
$$ b\le a  \le \sum_{n\ge 1}\mu\left(\frac{1}{2^{n}};a\right )\chi_{[\frac{1}{2^n}, \frac{1}{2^{n-1}})}  =\sigma_2(b ). $$
Therefore, $b \in E(0,1)^{\times \times}$ but $b\notin E(0,1)$.

 Now, we define a
   self-adjoint element $y$ by setting
 $$z:=  \sum_{n\ge 1} a_n\left( p_{n1}- p_{n2}\right)  , ~a_n =\mu\left(
 \frac{1}{2^n}; a\right) ,~n\ge 1, $$
 where $p_{n1}$ and $p_{n2} $ are mutually orthogonal non-trivial projections in the $n$-th direct summand $\mathbb{M}_2$, and the series is taken in the measure topology.
 In particular, $\mu(z)= \sigma_2(b)$ and therefore,
  $z\in E(\cM,\tau)^{\times\times}$ but $z\notin E(\cM,\tau)$.

  We claim that $\delta_z(\cA)\subset E(\cM,\tau)_b$, the closure of $ \cM$ in $E(\cM,\tau)$.
  Indeed, since $\delta_z$ has image in $E^{\times\times}(\cM,\tau)$ (generated by $E(0,1)^{\times\times }$), it follows from
  Ringrose's theorem \cite{Ringrose} that $\delta_z$ is bounded from
  $\left(\cA,\norm{\cdot}_\infty \right)$ into $\left(
  E^{\times\times } (\cM,\tau),\norm{\cdot}_{E^{\times \times }}\right)$.
Let $p_n$ be the identity of the $n$-th direct summand $\mathbb{M}_2$.
 For any element $y\in \cA$, we have $y_n:=y \sum_{k=1}^n p_k \to y$
in the uniform norm topology.
 Note that $y_n z,zy_n\in \cM$. 
 Hence, for any $n>m$, we have
\begin{align*}\norm{ \delta_z(y_n)-\delta_z (y_m) }_{E(\cM,\tau) }
&=
\norm{\mu(\delta_z(y_n)-\delta_z (y_m))}_{E(0,1)}\\&
\stackrel{\eqref{bidualequ}}{=}\norm{\mu(\delta_z(y_n)-\delta_z (y_m))}_{E(0,1)^{\times\times}}\\
&=\norm{ \delta_z(y_n)-\delta_z (y_m) }_{E ^{\times\times}(\cM,\tau)}\\
&\le 2 \norm{y_n-y_m}_\infty \norm{z}_{E ^{\times\times}(\cM,\tau)}
 \to 0\mbox{ as $n\to \infty$.}
 \end{align*}
 Therefore, $\left( \delta_z(y_n)\right)_{n\ge 1}$ is a Cauchy sequence in $E(\cM,\tau)$, which implies that $\left( \delta_z(y_n)\right)_{n\ge 1}$ converges to some element in $E(\cM,\tau)$ (in particular,  $\left( \delta_z(y_n)\right)_{n\ge 1}$ is $t_\tau$-convergent). On the other hand, $\delta_z(y_n)\to _{t_\tau}\delta_z(y)$ as $n\to \infty$.
 Therefore,  $\norm{ \delta_z(y_n)_{n\ge 1}-\delta_z(y)}_{E(\cM,\tau)}\to _n 0$.
Hence, $\delta_z(y)\in  E(\cM,\tau)$ for all $y\in \cA$.

%


 Assume that there exist $z_1\in E(\cM,\tau)$ such that $\delta_{z_1}=\delta_z$, i.e., $z_1-z $ commutes with $\cA$.
 That is, $z_1 -z=\sum_{n\ge 1} b_n p_n$ for some sequence $(b_n)_{n\ge 1}$ of complex numbers, where the series converges  in the measure topology.
 However, \begin{align*}
  E(0,1)\ni \mu(z_1)= \mu\left(z+ \sum_{n\ge 1} b_n p_n\right) &=    \mu\left( \sum_{n\ge 1} a_n\left( p_{n1}- p_{n2}\right)  + \sum_{n\ge 1} b_n p_n\right)\\
   &=    \mu\left( \sum_{n\ge 1} (a_n +b_n) p_{n1}- (a_n-b_n) p_{n2}  \right)\\
   &\ge    \mu\left( \sum_{n\ge 1} \max\{ |a_n +b_n|, |a_n-b_n|\} p_{n1}  \right)\\
   &\ge \mu\left(\sum _ {n\ge 1} a_n p_{n1}\right)= \mu\left(
  \sum _{n\ge 1} a_n p_{n2}\right)=\sigma_{1/2}\mu(z ),
  \end{align*}
which implies that $z\in E(\cM,\tau)$\cite[Theorem II.4.4]{KPS}.
This is a contradiction to the assumption.
 \end{proof}
%

The following result is an immediate consequence of Corollaries  \ref{cormain1}, \ref{out bounded} and Example \ref{out01}.
 \begin{corollary}\label{cormain}
 For a given symmetric function space $E(0,\infty)\subset S_0(0,\infty )$, the following two statements are equivalent:
 \begin{enumerate}
   \item for any   von Neumann algebra $\cM$ equipped with a semifinite faithful normal trace $\tau$ and any $C^*$-subalgebra $\cA$ of $\cM$, derivations $\delta:\cA\to E(\cM,\tau)$ are necessarily inner;
   \item $E(0,\infty)$ has the Levi property.
 \end{enumerate}

 \end{corollary}

The following result is an immediate consequence of Corollary  \ref{cormain1} and  Theorem \ref{noFatou}.
 \begin{corollary}

 Let $C_E$ be a symmetric ideal of compact operators in $B(\cH)$.
 Then, the  following two statements are equivalent:
 \begin{enumerate}
   \item for  any  $C^*$-subalgebra $\cA$  of $B(\cH)$, derivations $\delta:\cA\to C_E $ are necessarily inner;
   \item the commutative core $\ell_E$ of $C_E$ has the Levi property.
 \end{enumerate}

 \end{corollary}

\end{document}